\documentclass{amsart}
\usepackage[psamsfonts]{amssymb}
\usepackage{amsmath,amsthm}

\usepackage{enumerate}
\usepackage{multirow}

\newtheorem{thm}{Theorem}[section]
\newtheorem{cor}[thm]{Corollary}
\newtheorem{lem}[thm]{Lemma}
\newtheorem{prop}[thm]{Proposition}

\newtheorem{defn}[thm]{Definition}

\theoremstyle{remark}

\def\sph{\mathbb{S}^{d-1}}
\def\f{\frac}

 \def\a{{\alpha}} 
 \def\b{{\beta}}

 \def\l{{\lambda}}
 
 \def\o{{\omega}}
 \def\s{{\sigma}}
 \def\la{{\langle}}
 \def\ra{{\rangle}}

 \def\CD{{\mathcal D}}

 \def\CH{{\mathcal H}}

 \def\CP{{\mathcal P}}

 \def\CV{{\mathcal V}}
 
 \def\BB{{\mathbb B}}

 \def\NN{{\mathbb N}}

 \def\RR{{\mathbb R}}
  \def\SS{{\mathbb S}}

\newcommand{\wt}{\widetilde}

\begin{document}
 
\title[Sobolev orthogonal polynomials on the unit ball]
{Weighted Sobolev orthogonal polynomials on the unit ball}
\author[T. E. P\'erez]{Teresa E. P\'erez}
\address[T. E. P\'erez]{{Departamento de Matem\'atica Aplicada\\
Universidad de Granada\\
18071 Granada, Spain}}\email{tperez@ugr.es}
\thanks{{The work of the first and second authors was supported in part 
        by Ministerio de Ciencia e Innovaci\'on (Micinn)
        of Spain and by the European Regional Development Fund
        (ERDF) through the grant MTM2011-28952-C02-02}}

\author[M. A. Pi\~{n}ar]{Miguel A. Pi\~{n}ar}
\address[M. A. Pi\~{n}ar]{{Departamento de Matem\'atica Aplicada\\
Universidad de Granada\\
18071 Granada, Spain}}\email{mpinar@ugr.es}

\author[Yuan Xu]{Yuan Xu}
\address[Yuan Xu]{Department of Mathematics\\ University of Oregon\\
    Eugene, Oregon 97403-1222.}\email{yuan@math.uoregon.edu}
\thanks{The work of the third author was supported in part by 
NSF Grant DMS-1106113.}

\date{\today}

\begin{abstract}
For the weight function $W_\mu(x) = (1-\|x\|^2)^\mu$, $\mu > -1$, $\lambda  > 0$ {and 
$b_\mu$ a normalizing constant}, a family of mutually 
orthogonal polynomials on the unit ball with respect to the inner product 
$$
 \la f ,g \ra =  {b_\mu \left[\int_{\BB^d} f(x) g(x)  W_\mu(x)  dx +
    \lambda  \int_{\BB^d} \nabla f(x)\,\cdot \nabla g(x) W_\mu(x) dx\right]}
$$ 
are constructed in terms of spherical harmonics and a sequence of Sobolev orthogonal polynomials 
of one variable. The latter ones, hence, the orthogonal polynomials with respect to $\la \cdot,\cdot\ra$,
can be generated through a recursive formula. 
\end{abstract}

\maketitle

\section{Introduction}
\setcounter{equation}{0}
 
The orthogonal polynomials on the unit ball  $\BB^d$ of $\RR^d$ have been well understood, 
especially those that are orthogonal with respect to the inner {product} 
\begin{equation} \label{eq:ordinary-ip}
  \la f, g\ra_\mu : = {b_\mu}\int_{\BB^d} f(x) g(x) W_\mu(x) dx,
\end{equation}
where $W_\mu(x): = (1-\|x\|^2)^\mu$ on $\BB^d$, $\mu > -1$, {and 
$b_\mu$ is a normalizing constant such that $\la 1 ,1 \ra_\mu = 1$}. The condition $\mu > -1$ is necessary
for the inner product to be well defined. The orthogonal polynomials for $\la f, g \ra_\mu$ are often called 
classical and they are eigenfunctions of a second order linear differential operator ${D_\mu}$. 

The purpose of this paper is to study orthogonal polynomials on the unit ball with respect to an 
inner product that involves the integral of $\nabla f \cdot \nabla g$. The study in this direction was 
initiated in \cite{Xu08}, where the inner product
\begin{equation} \label{eq:first-ip}
  \la f,g \ra_{I} : = {\frac{\l}{\o_d}} \int_{\BB^d} \nabla f(x) \cdot \nabla g(x) dx + {\frac{1}{\o_d}} \int_{\sph} f(x) g(x) d\s, \quad \l  >0, 
\end{equation}
and another one that has the last term in the right hand side of $\la f, g\ra_I$ replaced by ${f(0)g(0)}$
were considered. The orthogonal polynomials with respect to the inner product $\la f, g \ra_I$ have the 
distinction that  they are eigenfunctions of the second order partial differential operator ${D_{-1}}$, which is 
the limiting case of ${D_\mu}$ when $\mu \to -1$ (\cite{PX08}). For the inner product $\la f, g \ra_I$, the main 
term is the first term, the integral over $\BB^d$, which however is zero if both $f$ and $g$ are constant, and
the second term is necessary to make sure that $\la f , f \ra_I =0$ implies $f =0$ almost everywhere.  

In the present paper, we first consider an extension of $\la f, g\ra_I$ that has the weight function 
$W_\mu$ in the integral over $\BB^d$ and will construct an orthonormal basis for this extension. 
We then study orthogonal polynomials with respect to the inner product
\begin{equation} \label{eq:main-ip}
 \la f ,g \ra_{\mu,\BB^d}
     =  {b_\mu \left[\int_{\BB^d} f(x) g(x)  W_\mu(x)  dx +
    \lambda  \int_{\BB^d} \nabla f(x)\,\cdot \nabla g(x) W_\mu(x) dx\right]},
\quad \l > 0.
\end{equation}  
In this case, both the first term and the second term are over the unit ball and it 
{is} no longer clear which
one is the dominating term. Nevertheless, we shall construct a sequence of orthogonal polynomials 
with respect to $ \la f ,g \ra_{\mu,\BB^d}$ explicitly, which depends on a two--parameter family of Sobolev 
polynomials of one variable. The latter ones can be expressed as a sum of Jacobi polynomials with 
the coefficients computed by a recursive relation. The norm of the orthogonal polynomials with 
respect to $ \la f ,g \ra_{\mu,\BB^d}$ can also be computed by a simple recursive relation. Such an
orthogonal basis can be useful in the spectral method for PDE \cite{Li}. 

In contrast to the case of one variable, the Sobolev orthogonal polynomials in several variables have
been studied only in a few cases and mostly on the unit ball \cite{AH, PI, PX08, Xu06, Xu08}. The setting 
of the unit ball has turned out to be more 
accessible and has revealed several interesting phenomena, which has stimulated and provided hint 
for the structure of the Sobolev orthogonal polynomials on other domains (see \cite{AX}). The present 
work is a further study on the unit ball, the results illustrate the impact of the secondary term in the 
Sobolev inner product, and also reveal further properties of the classical orthogonal polynomials on 
the unit ball. 

The paper is organized as follows. In the next section, we state the background materials and 
prove several properties of orthogonal polynomials on the unit ball and spherical harmonics that
involve the action of $\nabla$. Orthogonal polynomials with respect to the Sobolev inner product 
that extends \eqref{eq:first-ip} to weighted setting are discussed in Section 3, while those that are 
orthogonal with respect to $\la \cdot, \cdot \ra_{\mu, \BB^d}$ in \eqref{eq:main-ip} are discussed in 
Section 5. The Sobolev orthogonal polynomials of one variable that will be needed for the construction 
in the Section 5 are studied in Section 4. 

\section{Orthogonal polynomials on the ball and spherical harmonics}
\setcounter{equation}{0}

In this section we describe background materials on orthogonal polynomials and spherical harmonics 
that we shall need. The first subsection recalls the basic results on the classical orthogonal polynomials 
on the unit ball. The second subsection contains several lemmas on the spherical harmonics under the
action of the gradient operator. The third subsection collects properties on the Jacobi polynomials that 
we shall need later. 

\subsection{Orthogonal polynomials on the unit ball}
For $x,y  \in \RR^d$, we use the usual notation of $\|x\|$ and $\la x,y \ra$ to denote the Euclidean 
norm of $x$ and the dot product of $x,y$.  The unit ball and the unit sphere in $\mathbb{R}^d$ are 
{denoted}, respectively, by
$$
\BB^d :=\{x\in \mathbb{R}^d: \|x\| \le 1\} \qquad \textrm{and} \qquad \SS^{d-1}:=\{\xi\in \mathbb{R}^d: \|\xi\| = 1\}.
$$
For $\mu \in \RR$, let $W_\mu$ be the weight function defined by
$$
    W_\mu(x) = (1-\|x\|^2)^\mu, \qquad  \|x\| < 1.
$$
The function $W_\mu$ is integrable on the unit {ball} if $\mu > -1$, for which we denote the normalization 
constant of $W_\mu$ by ${b_\mu}$, 
$$
b_\mu := \left(\int_{{\BB^d}}\, W_\mu(x) \, dx\right)^{-1} = \frac{\Gamma(\mu + d/2 + 1)}{\pi^{d/2}\Gamma(\mu+1)}.
$$
The classical orthogonal polynomials on the unit ball are orthogonal with respect to the inner product
\begin{equation}\label{ball-ip}
   \la f,g \ra_\mu = b_\mu \int_{{\BB^d}}\, f(x)\, g(x) \, W_\mu(x) \, dx,
\end{equation}
which is normalized so that $\la 1,1\ra_\mu = 1$. 

Let $\Pi^d$ denote the space of polynomials in $d$ real variables. For $n = 0,1,2,\ldots,$
let $\Pi_n^d$ denote the linear space of polynomials in several variables of (total) degree 
at most $n$ and let $\CP_n^d$ denote the space of homogeneous polynomials of degree 
$n$. It is well known that 
$$
  \dim \Pi_n^d = \binom{n+d}{n} \quad \hbox{and} \quad \dim \CP_n^d = \binom{n+d-1}{n}:= r_n^d.
$$
A polynomial $P \in \Pi_n^d$ is called orthogonal with respect to $W_\mu$ on the ball if 
$\la P, Q\ra_\mu =0$ for all $Q \in \Pi_{n-1}^d$, that is, if it is orthogonal to all polynomials
of less {degree}. Let $\CV_n^d(W_\mu)$ denote the space of orthogonal polynomials of total
degree $n$ with respect to $W_\mu$. Then $\dim \, \mathcal{V}_n^d(W_\mu) = r_n^d.$

For $n\ge 0$, let $\{P^n_{\alpha}(x) : |\alpha|=n\}$ denote a basis of $\mathcal{V}_n^d(W_\mu)$. 
Then every element of $\CV_n^d(W_\mu)$ is orthogonal to polynomials of less degree. If the 
elements of the basis are also orthogonal to each other, that is, $\la P_\a^n, P_\beta^n \ra_\mu=0$ 
whenever $\a \ne \b$, we call the basis mutually orthogonal. If, in addition,  
$\la P_\a^n, P_\a^n \ra_\mu =1$, we call the basis orthonormal. 


In spherical--polar coordinates $x = r \xi$, $ r > 0$ and $\xi \in \sph$, a mutually orthogonal basis 
of $\CV_n^d(W_\mu)$ can be given in terms of the Jacobi polynomials and spherical harmonics. 

Harmonic polynomials of $d$-variables are homogeneous polynomials in $\CP_n^d$ that satisfy 
the Laplace equation $\Delta Y = 0$, where 
$\Delta = \f {\partial^2}{\partial x_1^2} + \ldots + \f{\partial^2}{\partial x_d^2}$
is the usual Laplace operator. Let $\mathcal{H}_n^d$ denotes the space of harmonic polynomials
of degree $n$. It is well know that
$$
         a_n^d: = \dim \mathcal{H}_n^d = \binom{n+d-1}{n} - \binom{n+d-3}{n}.
$$
Spherical harmonics are the restriction of harmonic polynomials on the unit sphere. If $Y \in
\mathcal{H}_n^d$, {then} $Y(x) = r^n Y(\xi)$ in spherical--polar {coordinates} 
$x = r \xi$, so that 
$Y$ is uniquely determined by its restriction on the sphere. We shall also use $\CH_n^d$ to 
denote the space of spherical harmonics of degree $n$. Let $d \s$ denote the surface measure
and $\o_d$ denote the surface area, 
$$
  \omega_d := \int_{\SS^{d-1}} d\s = \frac{2\, \pi^{d/2}}{\Gamma(d/2)}.
$$
The spherical harmonics of different degrees are orthogonal with respect to the inner product
$$
   \la f, g \ra_{\sph}: = \f{1}{\o_d} \int_{\sph} f(\xi) g(\xi) d\s(\xi). 
$$

Let $P_n^{(\a,\b)}(t)$ denotes the usual Jacobi polynomial of degree $n$, which is orthogonal with
respect to the weight function 
$$
   w_{\a,\b}(t) = (1-t)^\a (1+t)^\b, \qquad \a,\b > -1, \quad t \in [-1,1].
$$
We can now state the orthogonal basis of $\CV_n^d(W_\mu)$ in spherical--polar coordinates (see, for 
example, \cite{DX01}). 

\begin{lem}
For $n \in \NN_0$ and $0 \le j \le n/2$, let $\{Y_\nu^{n-2j}: 1\le \nu\le a_{n-2j}^d\}$ denote
an orthonormal basis for $\mathcal{H}_{n-2j}^d$. Define 
\begin{equation}\label{baseP}
P_{j,\nu}^{n}(x) = P_{j}^{(\mu, n-2j + \frac{d-2}{2})}(2\,\|x\|^2 -1)\, Y_\nu^{n-2j}(x). 
\end{equation}
Then the set $\{P_{j,\nu}^{n}(x): 1 \le j \le n/2, \,1 \le \nu \le a_{n-2j}^d \}$ is a mutually
orthogonal basis of $\CV_n^d(W_\mu)$. More precisely, 
$$
\la P_{j,\nu}^{n}(x), P_{k,\eta}^{m}(x)\ra_\mu =  H_{j,n}^{\mu}  \delta_{n,m}\,\delta_{j,k}\,\delta_{\nu,\eta},
$$
where $ H_{j,n}^{\mu}$ is given by 
\begin{equation} \label{eq:Hjn-mu}
 H_{j,n}^{\mu}: = \frac{(\mu +1)_j (\frac{d}{2})_{n-j} (n-j+\mu+ \frac{d}{2})}
    { j! (\mu+\frac{d+2}{2})_{n-j} (n+\mu+ \frac{d}{2})}, 
\end{equation} 
{where $(a)_n= a(a+1) \ldots (a+n-1)$ denotes the Pochhammer symbol.} 
\end{lem}
It is known that orthogonal polynomials with respect to $W_\mu$ are eigenfunctions of a second order
differential operator $\CD_\mu$. More precisely, we have
\begin{align} \label{eq:Bdiff}
    \CD_\mu P =  -(n+d) (n + 2 \mu)P, \qquad \forall P \in \CV_n^d(W_\mu), 
\end{align} 
where 
\begin{equation*}
  \CD_\mu := \Delta  - \sum_{j=1}^d \frac{\partial}{\partial x_j} x_j \left[
  2 \mu  + \sum_{i=1}^d x_i \frac{\partial  }
  {\partial x_i} \right].
\end{equation*}

\subsection{Spherical harmonics}
For later study, we will need several properties of spherical harmonics. 
Let $\nabla$ denote the gradient operator
$$
     \nabla f := (\partial_1 f, \partial_2 f, \ldots, \partial_d f)^T,
$$
where $\partial_k$ denotes the partial derivative in the $k$--th variable. We denote the dot product of 
$\nabla f$ and $\nabla g$ by $\nabla f \cdot \nabla g$.

In spherical--polar coordinates $x = r \xi$, $r \ge 0$ and $\xi \in \sph$, the differential operators $\nabla$
and $\Delta$ can be decomposed as follows (cf. \cite{DX12}): 
\begin{align}
\nabla = & \frac{1}{r}\nabla_0 + \xi \frac{\partial}{\partial r}, \label{nabla-radial} \\
\Delta = & \frac{\partial^2}{\partial r^2} + \frac{d-1}{r}\frac{\partial}{\partial r} + \frac{1}{r^2}\Delta_0. \label{L-B-operator}
\end{align}
The operators $\nabla_0$ and $\Delta_0$ are the spherical parts of the gradient and the Laplacian,
respectively. The operator $\Delta_0$ is called the Laplace--Beltrami operator, which has spherical harmonics as
its eigenfunctions. More precisely, it holds(cf. \cite{DX12})
\begin{equation}\label{eigen}
\Delta_0 Y (\xi) = -n(n+d-2)Y(\xi), \quad \forall \,Y\, \in\,\mathcal{H}_n^d, \quad \xi \in \SS^{d-1}.
\end{equation}

In addition, harmonic polynomials satisfy the following properties:
 
\begin{lem} \label{lem:2.2-harmonic}
Let $\{Y_\nu^n: 1 \le \nu \le a_n^d\}$ be an orthonormal basis of $\CH_n^d$. Let $x = r \xi$, with $r > 0$
and $\xi\in \SS^{d-1}$. Then
\begin{enumerate}[\rm (i)]
\item $\xi \cdot\nabla_0 Y_\nu^{n}(x)= 0,$
\item $\displaystyle{ \nabla Y_\nu^{n}(x)\cdot \nabla Y_\eta^{m}(x) =
    \frac{1}{r^2}\,\nabla_0 Y^{n}_\nu(x)\cdot \nabla_0 Y^{m}_\eta (x) 
 +\frac{n\,m}{r^2} Y^{n}_\mu(x)Y^{m}_\eta(x)}$.

\item For $1\le \nu \le a_n^d$ and $1\le \eta \le a_m^d$, the following relation holds
\begin{equation}\label{nabla-Y-int}
  \frac{1}{\o_d} \int_{\SS^{d-1}}\nabla Y_\nu^{n}(\xi)\cdot \nabla Y_\eta^{m}(\xi) d\s (\xi)
     = n(2n+d-2) \delta_{n,m}\delta_{\nu,\eta}. 
\end{equation}
\end{enumerate}
\end{lem}

\begin{proof} 
Since $Y_\nu^n$ is homogeneous, $\frac{d}{dr}Y_\nu^n(x) = n Y_\nu^n (x)/r$ and, by 
the Euler's equation for homogeneous polynomials, $x \cdot \nabla Y_\nu^n(x) = n Y_\nu^n(x)$.
Hence, by \eqref{nabla-radial} 
\begin{align*}
\xi \cdot \nabla_0 Y_\nu^{n}(x) 
 =  x \cdot \nabla Y_\nu^{n}(x) - r \frac{n}{r} Y_\nu^{n}(x) = n Y_\nu^{n}(x) - n Y_{{\nu}}^{n}(x)= 0.
\end{align*}

Using (i), the proof of (ii) follows from \eqref{nabla-radial} and a straightforward computation. 

For the proof of (iii), we will need the Green's formula on the sphere (cf. \cite{DX12}),
\begin{equation}\label{Green-Beltrami}
\int_{\SS^{d-1}} \nabla_0 f(\xi)\cdot\nabla_0 g(\xi) d\s(\xi) = - \int_{\SS^{d-1}} \Delta_0 f(\xi) g(\xi) d\s(\xi).
\end{equation}
Applying \eqref{Green-Beltrami} on the identity in (ii) and using \eqref{eigen}, we obtain
\begin{align*}
\lefteqn{\int_{\SS^{d-1}}\nabla Y_\nu^{n}(\xi)\cdot \nabla Y_\eta^{m}(\xi) d\s(\xi)}\\
 &=  \int_{\SS^{d-1}}\nabla_0 Y_\nu^{n}(\xi)\cdot \nabla_0 Y_\eta^{m}(\xi)d\s(\xi) + 
    nm \int_{\SS^{d-1}}Y_\nu^{n}(x)Y_\eta^{m}(x)d\s(\xi)\\
&= \int_{\SS^{d-1}}\Delta_0 Y_\nu^{n}(\xi)\,Y_\eta^{m}(\xi)d\s(\xi) + 
  n m \omega_{d}\delta_{n,m}\delta_{\nu,\eta}\\
&=  n(n+d-2)\int_{\SS^{d-1}} Y_\nu^{n}(\xi)Y_\eta^{m}(\xi)d\s(\xi) + 
n^2 \omega_{d}\delta_{n,m}\delta_{\nu,\eta}\\
&= n(2n+d-2) \o_{d}\delta_{n,m}\delta_{\nu,\eta}.
\end{align*}
This completes the proof. 
\end{proof}

\begin{cor} \label{lem:3.4}
Let $\{Y_\nu^n: 1 \le \nu \le a_n^d\}$ be an orthonormal basis of $\CH_n^d$. For $\mu > -1$, 
$n\neq m$, $1 \le \nu \le a_n^d$ and $1 \le \eta \le a_m^d$, 
\begin{enumerate}[ \rm(i)]
\item $\displaystyle{b_\mu \int_{\BB^d} Y_\nu^{n}(x)  Y_\eta^{m}(x)W_\mu(x) dx 
  = \frac{(\frac d  2)_n}{(\mu+1 +\frac d 2)_n}  \delta_{n,m} \delta_{\nu,\eta} }$
\medskip
\item
 $\displaystyle{b_\mu \int_{\BB^d} \nabla Y_\nu^{n}(x) \cdot \nabla Y_\eta^{m}(x)W_{\mu+1}(x) dx = 
    2n (\mu+1) \frac{(\frac d  2)_n}{(\mu+1 + \frac d 2)_n}  \delta_{n,m} \delta_{\nu,\eta}.}$
\end{enumerate}
\end{cor} 

\begin{proof}
The proof uses the following well known identity 
\begin{equation}\label{changevar}
        \int_{\BB^d} f(x) dx = \int_0^1 r^{d-1}\int_{\SS^{d-1}}f(r\,\xi)\, d\s (\xi)\,dr
\end{equation}
that arises from the spherical--polar coordinates $x=r\,\xi$, $\xi\in \SS^{d-1}$. Since 
$Y_\nu^n$ is a homogeneous polynomial, $Y_n^\mu(x) = r^n Y_n^\mu(\xi)$, so that
$$
b_\mu \int_{\BB^d} Y_\nu^{n}(x)  Y_\eta^{m}(x)W_\mu(x) dx 
  = b_\mu \o_d \int_0^1 r^{2n + d-1}(1-r^2)^\mu dr \delta_{n,m} \delta_{\nu,\eta} 
$$
by the orthonormality of $Y_\mu^n$ with respect to $\la f, g \ra_{\sph}$, from which 
(i) follows upon evaluating the integral and using the fact that 
\begin{equation}\label{b-mu-omega}
   b_\mu \o_d = 2 \frac{\Gamma(\mu+1+ \frac d 2)}{ \Gamma(\mu+1)\Gamma(\frac d 2)}. 
\end{equation}
Since $\nabla Y_\mu^n(x) = r^{n-1} (\nabla Y_\mu^n)(\xi)$, (ii) follows similarly. 
\end{proof}

\subsection{Classical Jacobi polynomials}
We collect the properties of the classical Jacobi polynomials $P_n^{(\a,\b)}(t)$ that we shall need
below, most of which can be found in \cite[chapt. 22]{AS} and \cite{Sz}. The Jacobi polynomial 
$P_n^{(\alpha, \beta)}(t)$ is normalized by 
\begin{equation} \label{jac-norm}
P_n^{(\alpha,\beta)}(1) = \binom{n+\alpha}{n} = \frac{(\a+1)_n}{n!},
\end{equation}
For $\a, \b > -1$, these 
polynomials are orthogonal with respect to the Jacobi weight function
$$
   w_{\a,\b}(t) := (1-t)^\alpha(1+t)^{\beta}, \qquad  -1< x < 1,
$$
and their $L^2$ norms in $L^2(w_{\a,\b}, [-1,1])$ are given by 
\begin{align}\label{normJ}
h_n^{(\alpha,\beta)} := & \int_{-1}^1 P_{n}^{(\alpha, \beta)}(t)^2 \, w_{\alpha,\beta}(t)\,dt 
 = \frac{2^{\alpha+\beta+1}}{2n+\alpha+\beta+1} \frac{\Gamma(n+\alpha+1)\,\Gamma(n+\beta+1)}{n!\,\Gamma(n+\alpha+\beta+1)}.
\end{align}
The polynomial $P_n^{(\alpha,\beta)}(t)$ is of degree $n$ and its leading coefficient $k_n^{(\alpha,\beta)}$
is given by 
\begin{equation}\label{leadingcoef}
k_n^{(\alpha,\beta)} = \frac{1}{2^n}\, \binom{2n + \alpha + \beta}{n}.
\end{equation}
The derivative of a Jacobi polynomial is again a Jacobi {polynomial}, 
\begin{equation}\label{derJ}
\frac{d}{d t} P_n^{(\alpha,\beta)}(t) = \frac{n+\alpha + \beta+1}{2} P_{n-1}^{(\alpha+1,\beta+1)}(t).
\end{equation}

The following relations between different families of the Jacobi polynomials can be found in 
(\cite[Chapt. 22]{AS}):
\begin{eqnarray}
&~& (1+t) P_n^{(\alpha,\beta+1)}(t) = \frac{n+\beta+1}{2n+\alpha+\beta+2}P_n^{(\alpha,\beta)}(t) + \frac{n+1}{2n+\alpha+\beta+2}P_{n+1}^{(\alpha,\beta)}(t),\label{RAF2}\\
&~& P_n^{(\alpha,\beta)}(t) = a_n^{(\alpha,\beta)} P_n^{(\alpha+1,\beta)}(t) - b_n^{(\alpha,\beta)}P_{n-1}^{(\alpha+1,\beta)}(t),\label{RAF}\\
&~& P_n^{(\alpha,\beta)}(t) = a_n^{(\alpha,\beta)} P_n^{(\alpha,\beta+1)}(t) + b_n^{(\beta,\alpha)}P_{n-1}^{(\alpha,\beta+1)}(t),\label{RAF0}
\end{eqnarray}
where
\begin{equation}\label{coef-a-b}
a_n^{(\alpha,\beta)} = \frac{n+\alpha+\beta+1}{2\,n + \alpha + \beta+1}, \quad b_n^{(\alpha,\beta)} = \frac{n+\beta}{2\,n + \alpha + \beta+1}.
\end{equation}
They are used to prove the following two relations that we need later. 

\begin{lem} \label{lem:Jacobi-recur}
For $\a > -1$ and $\b > 0$, 
\begin{align}
(1+t)\,\frac{d}{d t} P_n^{(\alpha,\beta)}(t) = \beta \, P_{n-1}^{(\alpha+1,\beta)}(t) + n \,P_n^{(\alpha+1,\beta-1)}(t)
   \label{jac-*} \\
\beta P_{n}^{(\alpha,\beta)}(t) + (1+t)\,\frac{d}{d t} P_n^{(\alpha,\beta)}(t) = (\beta + n) \,P_n^{(\alpha+1,\beta-1)}(t)
\label{jac-**}
\end{align}
\end{lem}

\begin{proof}
From \eqref{derJ} and \eqref{RAF2}, we deduce
\begin{align*}
(1+t)\,\frac{d}{d t} P_n^{(\alpha,\beta)}(t) =\ & (1+t)\,\frac{n+\alpha+\beta +1}{2}P_{n-1}^{(\alpha+1,\beta+1)}(t)\\
=\ & \frac{n+\alpha+\beta+1}{2n+\alpha+\beta+1} \left[(n+\beta)\,P_{n-1}^{(\alpha+1,\beta)}(t) + n\,P_n^{(\alpha+1,\beta)}(t)\right],
\end{align*}
replacing the last term by, according to \eqref{RAF0}, 
$$
P_{n}^{(\alpha+1,\beta)}(t) = \frac{2n+\alpha+ \beta +1}{n+\alpha + \beta +1}P_{n}^{(\alpha+1,\beta-1)}(t)-
\frac{n+\alpha + 1}{n+\alpha + \beta +1}P_{n-1}^{(\alpha+1,\beta)}(t)
$$
and simplifying the result, we obtain \eqref{jac-*}. The relation \eqref{jac-**} is deduced from 
\eqref{jac-*} upon using $P_n^{(\alpha,\beta-1)}(t) - P_n^{(\alpha-1,\beta)}(t) = P_{n-1}^{(\alpha,\beta)}(t)$,
which follows from \eqref{RAF} and \eqref{RAF0}.
\end{proof}

\section{Sobolev orthogonal polynomials with a spherical term}
\setcounter{equation}{0}

The first study of such orthogonal polynomials was initiated in \cite{Xu08}, where 
orthogonal polynomials with respect to the inner product 
\begin{equation}\label{ip1A}
  \langle f,g \rangle_\nabla :=  \frac{\l}{\o_d}  \int_{\BB^d} 
     \nabla  f(x) \cdot \nabla  g(x) dx + \frac{1}{\o_{d}} \int_{\SS^{d-1}} f(\xi)g(\xi) d\s(\xi)
\end{equation}
were studied. Let $\CV_n^d (\nabla)$ denote the space of orthogonal polynomials of degree $n$
with respect to $\la \cdot,\cdot \ra_\nabla$. An orthogonal basis for this inner product is given 
as follows. 

 \begin{thm} \label{thm:U-basis}
A mutually orthogonal basis $\{U_{j,\nu}^n: 0 \le j \le \frac{n}{2}, 
1 \le \nu \le {a_{n-2j}^d} \}$ for $\CV_n^d(\nabla)$ is given by
\begin{align}\label{eq:basisI}
\begin{split}
 U_{0,\nu}^n(x) & = Y_\nu^n(x), \quad \\
 U_{j,\nu}^n(x) & = (1-\|x\|^2) P_{j-1}^{(1,n-2j+\frac{d-2}{2})}(2\|x\|^2-1)
     Y_\nu^{n-2j}(x), \quad 1 \le j \le \frac{n}{2},
\end{split}
\end{align}
where $\{Y_\nu^{n-2j}: 1 \le \nu \le {a_{n-2j}^d}\}$ is an orthonormal basis 
of $\CH_{n-2j}^d$. Furthermore, 
\begin{align}\label{eq:RnormI}
\langle U_{0,\nu}^n, U_{0,\nu}^n \rangle_\nabla = n \l +1, \qquad
\langle U_{j, \nu}^n, U_{j, \nu}^n \rangle_\nabla =\frac{2 j^2}{n+\frac{d-2}{2}} \l.
\end{align}
\end{thm}

In particular, it was observed that the space $\CV_n^d(\nabla)$ can be decomposed in 
terms of the space $\CV_n(W_\mu)$ of classical orthogonal polynomials on the ball
and the space $\CH_n^d$ of the spherical harmonics; that is,  for $n \ge 1$, 
\begin{equation} \label{eq:Vn-decomp}
  \CV_n^d(\nabla) = \CH_n^d \oplus (1-\|x\|^2)\CV_{n-2}^d(W_1).
\end{equation}
{Moreover}, it was shown in \cite{PX08} that the elements of $\CV_n^d(\nabla)$ are 
eigenfunctions of the differential operator $\CD_{-1}$, that is, they satisfy the equation 
\eqref{eq:Bdiff} with $\mu = -1$. The phenomenon of the decomposition \eqref{eq:Vn-decomp}
was first observed in \cite{Xu06} for orthogonal polynomials with respect to an inner product 
that involves second order partial derivatives. 

The proof that $U_{j,\nu}^n$ are orthogonal with respect to $\la \cdot, \cdot\ra_\nabla$ given in
\cite{Xu08} is based on the Green's identity  
$$
\int_{\BB^d} \nabla f(x)\cdot \nabla g(x)\, dx = \int_{\SS^{d-1}} f(\xi)\, \frac{d}{dr} g(\xi)\,d\s(\xi)
 - \int_{\BB^d} f(x) \,\Delta g(x) dx.
$$
There is a weighted extension of Green's formula, which states that 
\begin{align}\label{eq:weighted-Green}
\int_{\BB^d} \nabla f(x)\cdot \nabla g(x)\, h(x) dx =& \int_{\SS^{d-1}} f(\xi)\, \frac{d g}{dr} (\xi)\, h(\xi)\,d\s \\
& - \int_{\BB^d} f(x) \left(\Delta g(x)\,h(x) + \nabla g(x)\,\nabla h(x)\right) dx, \nonumber
\end{align}
where $f,g,h$ are differentiable functions for which the integrals are finite. The proof of this identity 
follows from applying the divergence theorem on the identity 
$$
{\rm div~}(f\,\nabla g\, h) =   \nabla f\cdot \nabla g\, h + f\,\Delta g\,h + f\,\nabla g\cdot\nabla h,
$$
where ${\rm div}$ is the divergence operator defined by ${\rm div} (f_1(x), f_2(x), \ldots, f_d(x))^T 
= \sum_{i=1}^d \partial_i f_i(x)$. 

As an extension of \eqref{ip1A}, one can consider more generally the weighted inner product 
defined as follows.

\begin{defn}
For $\mu > -1$ and $f,g \in \Pi^d$, {we} define
\begin{equation} \label{eq:ip-sph-W}
   \la f, g \ra_{\nabla, W_\mu}:=  \frac{\l}{\o_d}  \int_{\BB^d} 
     \nabla  f(x) \cdot \nabla  g(x) W_{\mu+1}(x) dx + \frac{1}{\o_{d}} \int_{\SS^{d-1}} f(\xi)g(\xi) d\s(\xi).
\end{equation}
\end{defn}

It is easy to see that this defines an inner product. The inner product $\la f,g\ra_\nabla$ in \eqref{ip1A} 
is the limiting case when $\mu = -1$. 

Let us denote by $\CV_n^d(\nabla, W_\mu)$ the space of orthogonal polynomials of degree exactly
$n$ with respect to this inner product. We want to construct a basis for this space. 

Recall that the polynomials ${P_{j,\nu}^n(x)}$, defined in \eqref{baseP},  are orthogonal 
with respect to $W_\mu$ on $\BB^d$. To construct an orthogonal basis for $\CV_n^d(\nabla, W_\mu)$,
we need the following proposition that is of independent interest, which states that the orthogonal 
polynomials $P_{j,\nu}^n(x;W_\mu)$ satisfy another orthogonality on the ball. 

\begin{prop} \label{lem:3.3}
Let $\mu > -1$ and let  $P_{j,\nu}^n(x)$ be the mutually orthogonal polynomials in $\CV_n^d(W_\mu)$, 
defined in \eqref{baseP}. Then 
\begin{equation}\label{ortho-nabla}
b_\mu \int_{\BB^d} \nabla P_{j,\nu}^n (x) \cdot \nabla P_{k,\eta}^m (x) W_{\mu+1}(x) dx = 
   H_{j,n}^\mu(\nabla) \delta_{n,m} \delta_{j,k} \delta_{\nu,\eta},
\end{equation}
where, with $H_{j,\mu}^n$ defined in \eqref{eq:Hjn-mu},  
$$
   H_{j,n}^\mu(\nabla) : =  \left(n(2j+\mu+1)- j(2j-d+2) \right) H_{j,n}^\mu. 
$$
\end{prop}

\begin{proof}
For $\mu > -1$, $W_{\mu+1}(x) = (1-\|x\|^2) W_{\mu}(x)$ vanishes when $x \in \SS^{d-1}$. Hence, by 
\eqref{eq:weighted-Green}, 
\begin{align}\label{eq:proof-a}
 & \int_{\BB^d} \nabla P_{j,\nu}^n(x) \cdot \nabla P_{k,\eta}^m(x) W_{\mu+1}(x) dx \\
& \quad = - \int_{\BB^d} P_{k,\eta}^m(x) \left[(1-\|x\|^2)\Delta - 2(\mu+1)\langle x,\nabla\rangle \right] P_{j,\nu}^n(x)
 W_{\mu}(x) dx. \notag
\end{align}
Observe that $\left[(1-\|x\|^2)\Delta - 2(\mu+1)\langle x,\nabla\rangle \right] Q(x)$ has the same degree
as $Q$, this proves the orthogonality when $n \ne m$ by the orthogonality of $P_{k,\eta}^m \in 
\CV_m^d(W_\mu)$. To prove the {mutual} orthogonality in \eqref{ortho-nabla}, we need to compute
$$
  \left[(1-\|x\|^2)\Delta - 2(\mu+1)\langle x,\nabla\rangle \right] P_{j,\nu}^n(x).
$$
Let ${\beta_j} = n-2 j + \frac{d-2}{2}$. Then ${P_{j,\nu}^n (x)} = P_j^{(\mu,\b_j)} (2 r^2-1) r^{n-2j} Y_\nu^{n-2j}(\xi)$
in the spherical-polar coordinates $x = r \xi$. Using the expression of $\Delta$ in \eqref{L-B-operator}
and the derivative formula of the Jacobi polynomials, it is not difficult to verify that 
\begin{align*} 
   (1-\|x\|^2)\Delta P_{j,\nu}^n(x) = \, & 2(j+\mu+{\beta_j}+1) \left[ 2(j+\mu+{\beta_j}+2) r^2 P_{j-2}^{(\mu+2, {\beta_j}+2)}(2r^2-1) \right. \\
     &\left.+ (2n-4j+d)P_{j-1}^{(\mu+1, {\beta_j}+1)}(2r^2-1)\right] Y_\nu^{n-2j}(x).\notag
\end{align*}
Furthermore, by \eqref{nabla-radial}, $x \cdot \nabla = \xi \cdot \nabla_0 + r \frac{\partial}{\partial r}$ and, 
by (i) of Lemma \ref{lem:2.2-harmonic}, $\xi \cdot \nabla_0 Y_\mu^n(\xi) =0$. It follows that 
\begin{align*}
  \la x,\nabla \ra P_{j,\nu}^n(x)   = &  \left[ 2(j+\mu+{\beta_j}+1)r^2 P_{j-1}^{(\mu+1, {\beta_j}+1)}(2r^2-1) \right. \\
     & \left .+   (n-2j)P_{j}^{(\mu, {\beta_j})}(2r^2-1)\right] Y_\nu^{n-2j}(x). 
\end{align*}
Both of these terms are of the form $q(\|x\|^2) Y_\nu^{n-2j}(x)$. Hence, looking at the leading coefficients of
the Jacobi polynomials, given in \eqref{leadingcoef}, we conclude that 
\begin{align*}
& \left[(1-\|x\|^2)\Delta - 2(\mu+1)\langle x,\nabla\rangle \right] P_{j,\nu}^n(x) \\
&   \qquad = \left(n(2j+\mu+1)- j(2j-d+2) \right)P_j^{(\mu,{\beta_j})} (2 \|x\|^2-1) Y_\nu^{n-2j}(x) + g(x),
\end{align*}
where $g$ is a polynomial of degree at most $n-2$. Consequently, by the orthogonality of
$P_{j,\nu}^n$ as an element of $\CV_n^d(W_\mu)$, \eqref{ortho-nabla} follows from \eqref{eq:proof-a}.
\end{proof}

When $j = 0$, the  orthogonal polynomials ${P_{j,\nu}^n(x)}$ become spherical harmonics
$P_{0,\nu}^{n}(x) = Y_\nu^{n}(x)$, and \eqref{ortho-nabla} agrees with (ii) in Lemma \ref{lem:3.4}.

\begin{thm} \label{thm:3.5}
For $j \le n/2$, let $\{Y_\nu^{n-2j}: 0 \le \nu \le a_{n-2j}^d\}$ be an orthonormal basis of $\CH_{n-2j}^d$. 
Define  
\begin{align} \label{eq:Q-basis}
\begin{split}
Q_{0,\nu}^n(x) &: = {Y_\nu^n(x)}, \\
Q_{j,\nu}^n (x) & : = \left[P_j^{(\mu, \beta_j)}(2\|x\|^2 -1) - P_{j}^{(\mu, \beta_j)}(1) \right]Y_\nu^{n-2j}(x),
\quad 1 \le j \le \frac{n}2.
\end{split}
\end{align}
where $\beta_j : =n-2j+\frac{d-2}{2}$. Then $\{Q_{j, \nu}^n: 0 \le j \le n/2, 1 \le \nu \le a_{n-2j}^d\}$ 
forms a mutually orthogonal basis of $\CV_n^d(\nabla, W_\mu)$. Furthermore, 
\begin{align*}
 \la Q_{0,\nu}^n, Q_{0,\nu}^n \ra_{\nabla, W_\mu}  & =   \lambda n \frac{\Gamma(\mu+2)\Gamma(n+\f d 2)}
    {\Gamma(\mu+n+1 + \f d 2)} + 1, \\
  \la Q_{j,\nu}^n, Q_{j,\nu}^n \ra_{\nabla, W_\mu} & =  \lambda  \left(n(2j+\mu+1)- j(2j-d+2) \right) \\
    &  \qquad \times  \frac{\Gamma(\mu+j+1)\Gamma(n-j+\f d 2)
          (\mu+n-j+\frac{d}2)}{j! \Gamma(\mu+n- j+1+ \f d 2) (n+\mu+\frac{d}2)} \\
    & + \lambda (n-2j)\frac{ \Gamma(\mu+2) \Gamma(n-2j+\f d2)(\mu+1)_j^2}{\Gamma(n-2j+\mu + \f d 2)j^2}, 
        \quad \hbox{if $1 \le j \le \f n 2$}.
\end{align*}
\end{thm}
 
\begin{proof}
If $j =0$, then the orthogonality of $Q_{0,\nu}^n$ to lower degree polynomials and to $Q_{j,\nu}^n$ for 
$j > 0$ follows from (i) in Lemma \ref{lem:3.4} and the orthogonality of $Y_\nu^n$. Furthermore, the
same consideration shows that 
$$
\la Q_{0,\nu}^n, Q_{0,\nu}^n \ra_{\nabla, W_\mu} = \f {\l}{\o_d b_\mu}  2n (\mu+1)
  \frac{(\f d 2)_n}{(\mu+1+\f d 2)_n} + 1,  
$$
which simplifies to the stated formula upon using \eqref{b-mu-omega}.

If $0<  j  \le n/2$, then $Q_{j,\nu}^n \vert_{\sph}$ vanishes on the sphere as its radial part vanishes
when $\|x\| =1$, so that we only needs to consider the integral over $\BB^d$ in the inner product
$\la \cdot,\cdot\ra_{\nabla, W_\mu}$. Furthermore, in terms of $P_{j,\nu}^n$ in \eqref{baseP} of 
$\CV_n^d(W_\mu)$, we can write 
$$
  Q_{j,\nu}^n (x) = P_{j,\nu}^n(x) - P_j^{(\mu, n-2j + \f{d-2}{2})}(1) Y_\nu^{n-2j}(x) 
$$
and observe that $Y_\nu^{n-2j}$ is a polynomial of degree less than $n$ for $j > 0$, the mutual 
orthogonality of $\{Q_{j,\nu}^n: 0 \le j \le n/2\}$ then follows from \eqref{ortho-nabla}. Moreover,  
it follows readily that 
\begin{align*}
   \la Q_{j,\nu}^n, Q_{j,\nu}^n \ra_{\nabla, W_\mu} = & \frac{\lambda}{\o_d} 
     \left[ \int_{\BB^d} [\nabla P_{j,\nu}^n(x)]^2 W_{\mu+1}(x)dx  \right. \\
     &  \quad\, \, \left. + \frac{(\mu+1)_j^2}{j!^2} 
        \int_{\BB^d} [ Y_\nu^{n-2j} (x)]^2 W_{\mu+1}(x) dx   
    \right], 
\end{align*}
where we have used $ P_j^{(\mu,n-2j+\f{d-2}{2})}(1) =  \frac{(\mu+1)_j}{j!}$. Hence, the norm of 
$Q_{j,\nu}^n$ can now be deduced from \eqref{ortho-nabla} and (ii) of Lemma \ref{lem:3.4},
and {simplified} to the stated formula upon using \eqref{eq:Hjn-mu} and \eqref{b-mu-omega}. 
\end{proof}

If we set $\mu \to -1$ in the above theorem, then since \cite[(4.22.2)]{Sz}
$$
   P_j^{(-1,\beta)} (t) = \frac{j+\beta}{2j} (1-t) P_{j-1}^{(1,\beta)}(t), 
$$
it is easy to see that Theorem \ref{thm:3.5} agrees with Theorem \ref{thm:U-basis}. 
The decomposition \eqref{eq:Vn-decomp}, however, does not extend to $\CV_n^d(\nabla, W_\mu)$. 

Finally, it is worth to mention that the orthogonal polynomials in $CV_n^d(W_\mu)$ are automatically 
orthogonal with respect to a Sobolev inner product according to Proposition \ref{lem:3.3}. 

\begin{thm}
Let $\mu > -1$ and let  $P_{j,\nu}^n(x)$ be the mutually orthogonal polynomials in $\CV_n^d(W_\mu)$, 
defined in \eqref{baseP}. Then they are also mutually orthogonal with respect to the Sobolev inner
product
\begin{equation}\label{[f,g]mu}
[f,g]_\mu : =  b_\mu \left[ \int_{\BB^d}f(x) g(x) W_{\mu}(x) dx + \lambda \int_{\BB^d} \nabla f(x) \cdot \nabla g(x)
    W_{\mu+1}(x) dx \right], 
\end{equation}
where $\l > 0$ is a fixed constant. 
\end{thm}

Notice that the parameter of the weight function in the second integral is $\mu +1$. It is possible to add
terms with higher order derivatives with matching weight functions in the inner product, for which 
$P_{j,\nu}^n$ remain to be orthogonal. As we shall see in the Section 5, the orthogonal structure 
becomes far more complicated if we want the weight functions in the two terms of \eqref{[f,g]mu}
to be the same. 

\section{A family of Sobolev orthogonal polynomials of one variable}
\setcounter{equation}{0}
 
In the next section, we will construct an orthogonal basis for the Sovolev inner product \eqref{eq:main-ip}.
As in the cases of \eqref{baseP} and \eqref{eq:Q-basis}, the basis can be constructed in
{spherical}--polar coordinates,  where the main terms are of the form 
$$
  q_j(2\|x\|^2 -1) Y_\nu^{n-2j}(x).  
$$
The polynomial $q_j$ of one variable, however, is rather involved in this case, which we now define. 

\begin{defn}
Let $\lambda > 0$. For $\a > -1$ and  $\beta > \max\{0, \frac{d-2}2\}$, 
we define a {Sobolev inner} product by 
\begin{align}
    (f,g)_{\alpha, \beta} := & \int_{-1}^1 f(t) g(t)w_{\alpha,\beta}(t)dt \notag \\
 & + 2 \lambda \int_{-1}^1 (f,f')  
   \left( \begin{matrix} 
      A(\beta,d) & B(\beta,d) (1+t)\\
      B(\beta,d)\,(1+t) & 4 (1+t)^2
      \end{matrix} \right) \left(\begin{matrix}
      g\\  g'  \end{matrix}\right) w_{\alpha,\beta-1}(t)dt,\label{sob-1-v} 
\end{align}
where $A(\beta,d)=\beta B(\b, d)$ and $B(\beta,d)=2\beta-(d-2)$. 
\end{defn}

In the case of $\beta = \frac{d-2}{2}$, we have $A(\b,d)=B(\b,d)=0$ and the inner product takes a more
familiar form
$$
  (f,g)_{\alpha, \beta} = \int_{-1}^1 f(t) g(t)w_{\alpha,\beta}(t)dt + 8 \lambda \int_{-1}^1 f'(t) g'(t) w_{\alpha,\beta+1}(t)dt,
$$
{that is, a Sobolev inner product associated to a coherent pair of measures
(see \cite{Meijer}).}

We assume that $\beta > \max\{0, \frac{d-2}2\}$ so that the $2 \times 2$ matrix in \eqref{sob-1-v} is positive 
definite, which guarantees that $(f,g)_{\alpha, \beta}$ is indeed an inner product. Consequently, the orthogonal polynomials with respect to this inner product exist and they are uniquely determined up to a {multiplicative} constant. 

We let $q_{j}^{(\alpha,\beta)}(t)$ denote the orthogonal polynomial of degree $j$ with respect to $(f,g)_{\alpha,\beta}$
and normalize it so that it has the same leading coefficient as the Jacobi polynomial $P_j^{(\a,\b)}(t)$, that is,
\begin{equation}\label{qj-defn}
   q_{j}^{(\alpha,\beta)}(t) = k^{(\alpha, \beta)}_j\, t^j + \emph{lower degree terms}, 
      \quad k_j^{(\a,\b)} := \frac{1}{2^j} \binom{2j + \a+ \b}{j}. 
\end{equation}
With this normalization we have $q_{0}^{(\alpha,\beta)}(t) = P_{0}^{(\alpha,\beta)}(t)= 1$. 

Recall that the norm of the Jacobi polynomial ${P_j^{(\a,\b)}}$ is denoted by ${h_j^{(\a,\b)}}$, given in \eqref{normJ}. 
Let us denote the norm of $q_j^{(\a,\b)}$ in $(\cdot, \cdot)_{\a,\b}$ by $\wt h_j^{(\a,\b)}$, that is, 
$$
   \wt h_j^{(\a,{\b)}} : =  \left (q_{j}^{(\a,\beta)},q_{j}^{(\a,\beta)}\right)_{\a, \beta}.
$$
Our first result is to relate $q_j^{(\a,\b)}$ with the Jacobi polynomials. 

\begin{prop}
Let $\mu >-1$ and $\beta {>} \max\{0,(d-2)/2\}$. Then, for $j\ge 0$, 
\begin{equation}\label{1-2-one-v}
P_j^{(\mu, \beta)}(t) = a_j^{(\mu,\beta)} q_j^{(\mu+1,\beta)}(t) + d_{j-1}^{(\mu,\beta)}(\lambda)q_{j-1}^{(\mu+1,\beta)}(t),
\end{equation}
where $d_j^{(\a,\b)}$ are defined by 
\begin{equation}\label{def-d}
d_{-1}^{(\mu,\beta)} (\lambda)=0, \quad d_{j}^{(\mu,\beta)} (\lambda) = - b_{j+1}^{(\mu,\beta)} \frac{h_{j}^{(\mu+1,\beta)}}{\wt h_{j}^{(\mu+1,\beta)}},\quad j = 0, 1,2, \ldots, 
\end{equation}
and $a_j^{(\mu,\beta)}$ and $b_j^{(\mu,\beta)}$ are given by \eqref{coef-a-b}. 
\end{prop}

\begin{proof}
Since $\{q_k^{(\mu+1,\beta)}(t)\}_{k\ge0}$ is a sequence of orthogonal polynomials with 
respect to $(\cdot,\cdot)_{\mu+1,\b}$, we can expand $P_j^{(\mu,\beta)}(t)$ in terms of them, that is, 
$$
P_{j}^{(\mu, \beta)}(t) = \sum_{k=0}^j  d^j_k q_k^{(\mu+1,\beta)}(t) \quad \hbox{with}\quad
    d^j_k = \frac{(P_{j}^{(\mu, \beta)}, q_k^{(\mu+1,\beta)})_{\mu+1, \beta} } { \wt h_k^{(\mu +1,\beta)} }.
$$
Since $q_k^{(\a,\b)}$ is normalized so that its leading coefficient {is} the same as $P_j^{(\a,{\b})}$, it follows 
from \eqref{RAF} that  $d_j^j = a_j^{(\mu,\beta)}.$

We now {calculate} $d_k^j$ for $0 \le k < j$. From \eqref{sob-1-v}, we need to
compute 
\begin{align*}
   (P_{j}^{(\mu, \beta)}, q_k^{(\mu+1,\beta)})_{\mu+1,\beta}  = & \int_{-1}^1 P_{j}^{(\mu, \beta)}(t) q_k^{(\mu+1,\beta)}(t)
     w_{\mu+1,\beta}(t)dt  \\
& + {2} \lambda \left[A(\beta,d) \int_{-1}^1 P_{j}^{(\mu, \beta)}(t) q_k^{(\mu+1,\beta)}(t) w_{\mu+1,\beta-1}(t)dt\right.\\
& \quad + B(\beta,d) \int_{-1}^1 \frac{d}{d t}P_{j}^{(\mu, \beta)}(t) q_k^{(\mu+1,\beta)}(t) w_{\mu+1,\beta}(t)dt\\
& \quad + B(\beta,d) \int_{-1}^1 P_{j}^{(\mu, \beta)}(t) \frac{d}{d t}q_k^{(\mu+1,\beta)}(t) w_{\mu+1,\beta}(t)dt\\
& \quad + \left.4 \int_{-1}^1 \frac{d}{d t}P_{j}^{(\mu, \beta)}(t) \frac{d}{d t}q_k^{(\mu+1,\beta)}(t) w_{\mu+1,\beta+1}(t)dt\right].
\end{align*}
The first term in the right hand side vanishes for $k <j-1$ since, by \eqref{RAF} and the orthogonality 
of $P_j^{(\mu+1,\beta)}$, 
\begin{align*} 
 & \int_{-1}^1 P_{j}^{(\mu, \beta)}(t) q_k^{(\mu+1,\beta)}(t)w_{\mu+1,\beta}(t)dt \\
 &  \qquad\quad 
    = \int_{-1}^1 \left[a_j^{(\mu,\beta)} P_j^{(\mu+1,\beta)}(t) - b_j^{(\mu,\beta)}P_{j-1}^{(\mu+1,\beta)}(t)\right] 
     q_k^{(\mu+1,\beta)}(t) w_{\mu+1,\beta}(t) dt = 0,
\end{align*}
whereas for $k=j-1$, using $q_{j-1}^{(\a,\b)}(t) = P_{j-1}^{(\a,\b)}(t) + \ldots$ in addition shows that 
\begin{align*}
  & \int_{-1}^1 P_{j}^{(\mu, \beta)}(t) q_{j-1}^{(\mu+1,\beta)}(t) w_{\mu+1,\beta}(t)dt  \\
 & \qquad \quad = - b_j^{(\mu,\beta)}\int_{-1}^1 P_{j-1}^{(\mu+1,\beta)}(t) q_{j-1}^{(\mu+1,\beta)}(t)
     w_{\mu+1,\beta}(t)dt= - b_j^{(\mu,\beta)} h_{j-1}^{(\mu+1,\beta)}.
\end{align*}
An analogous reasoning, together with \eqref{derJ}, shows that the fourth and the fifth terms vanish 
for $0\le k\le j-1$. Thus, we only need to compute the second and the third terms. This is where 
Lemma \ref{lem:Jacobi-recur} is needed. Indeed, assume $\beta >0$ first; then, since $A(\b,d) = \b B(\b,d)$, 
the relation \eqref{jac-**} shows that, for $0\le k \le j-1$,
\begin{align*}
&A(\beta,d) \int_{-1}^1 P_{j}^{(\mu, \beta)}(t) q_k^{(\mu+1,\beta)}(t) w_{\mu+1,\beta-1}(t)dt\\
   & \qquad\quad  + B(\beta,d) \int_{-1}^1 \frac{d}{d t}P_{j}^{(\mu, \beta)}(t) q_k^{(\mu+1,\beta)}(t) w_{\mu+1,\beta}(t)dt\\ 
&=  [2\beta -(d-2)](\beta+j) \int_{-1}^1 P_{j}^{(\mu+1, \beta-1)}(t)\,q_k^{(\mu+1,\beta)}(t) w_{\mu+1,\beta-1}(t)dt=0, 
\end{align*}
whereas if ${\beta = \frac{d-2}{2}}$, then $A(\b,d)=B(\b,d)=0$ and these terms are automatically zero. 

Therefore, only the first term is nonzero and it is nonzero only for $k = j-1$. Consequently, 
$$
P_j^{(\mu, \beta)}(t) = a_j^{(\mu,\beta)}q_j^{(\mu+1,\beta)}(t) + d_{j-1}^{(\mu,\beta)}(\lambda)q_{j-1}^{(\mu+1,\beta)}(t),
$$
where $d_{j-1}^{(\mu,\beta)}(\lambda)$ is given in \eqref{def-d}. This completes the proof.
\end{proof}

If we knew how to compute $d_j^{(\mu,\b)}(\lambda)$, we could deduce $q_j^{(\mu+1,\b)}$ recursively 
by \eqref{1-2-one-v}. However, since $d_j^{(\mu,\b)}$ depends on the norm of $q_j^{(\mu+1,\b)}$, this has 
a ring of a futile loop. It turns out, as our next proposition shows, that $d_j^{(\mu,\b)}$ can be deduced from
a recursive relation. 

\begin{prop}
The coefficients $d_{j-1}^{(\mu,\beta)}(\lambda)$ satisfy a recurrence relation
\begin{align}
& d_0^{(\mu,\beta)}(\lambda) = -\frac{\beta+1}{(\mu+\beta+3)[1+\lambda (\mu +\beta +2)(\beta -(d-2)/2)]},\label{d-cero}\\
& d_{j}^{(\mu,\beta)}(\lambda) = -\frac{A_j^{(\mu,\beta)}}{[B_j^{(\mu, \beta)} + \lambda \, C_j^{(\mu, \beta)}]
   + d_{j-1}^{(\mu,\beta)}(\lambda)},\quad j\ge 1,\label{d-rec}
\end{align}
where
\begin{align}
& A_j^{(\mu,\beta)} = \frac{(j+\mu + \beta+1)(j+\mu+1)(j+\beta+1)(2j+\mu+\beta)}{j(2j+\mu+\beta+1)(2j+\mu+\beta+2)(2j+\mu+\beta+3)}, \label{A_j}\\
& B_j^{(\mu,\beta)} = 1+\frac{(\mu+\beta)(j+\mu+1)}{j(2j+\mu+\beta+2)},\label{B_j}\\
& C_j^{(\mu,\beta)} = \frac{(2j+\mu+\beta)[(\mu+1)(2\beta-(d-2)) + 4j(j+\mu + \beta+1)]}{2j}.\label{C_j}
\end{align}
\end{prop}

\begin{proof}
Since $q_{0}^{(\mu+1,\beta)}(x) =1$, the explicit expression for $d_0^{(\mu,\beta)}(\lambda)$ in \eqref{d-cero} 
is deduced directly from \eqref{def-d}, upon using \eqref{normJ}, \eqref{coef-a-b}, and the fact that 
$$
\wt h_{0}^{(\mu+1,\beta)} = (q_{0}^{(\mu+1,\beta)},q_{0}^{(\mu+1,\beta)})_{\mu+1, \beta}
     =  h_{0}^{(\mu+1,\beta)} + {2} \lambda  A(\beta,d) h_{0}^{(\mu+1,\beta-1)},
$$
which follows directly from the definition of $(\cdot,\cdot)_{\mu+1,\b}$. 

For $j\ge 1$, we use \eqref{1-2-one-v} and \eqref{def-d} to obtain
\begin{align*}
\left(P_j^{(\mu, \beta)}, P_j^{(\mu, \beta)}\right)_{\mu+1,\beta} &= \left(a_j^{(\mu,\beta)}\right)^2 
\left(q_j^{(\mu+1,\beta)},q_j^{(\mu+1,\beta)}\right)_{\mu+1,\beta}\\
& \qquad + \left(b_{j}^{(\mu,\beta)}\right)^2 \frac{(h_{j-1}^{(\mu+1,\beta)})^2}{(\wt h_{j-1}^{(\mu+1,\beta)})^2}
  \left (q_{j-1}^{(\mu+1,\beta)},q_{j-1}^{(\mu+1,\beta)}\right)_{\mu+1,\beta}\\
& = - \left (a_j^{(\mu,\beta)}\right )^2 b_{j+1}^{(\mu,\beta)} \frac{h_{j}^{(\mu+1,\beta)}}{d_{j}^{(\mu,\beta)}(\lambda)}
    - b_{j}^{(\mu,\beta)} h_{j-1}^{(\mu+1,\beta)} d_{j-1}^{(\mu,\beta)}(\lambda).
\end{align*}
Solving $d_{j}^{(\mu,\beta)}(\lambda)$ from the above identity leads to 
$$
 d_{j}^{(\mu,\beta)}(\lambda) = -\frac{ \left(a_j^{(\mu,\beta)}\right)^2\, b_{j+1}^{(\mu,\beta)} \,h_{j}^{(\mu+1,\beta)}}{\left(P_j^{(\mu, \beta)}, P_j^{(\mu, \beta)}\right)_{\mu+1,\beta} + b_{j}^{(\mu,\beta)} \, h_{j-1}^{(\mu+1,\beta)} \, d_{j-1}^{(\mu,\beta)}(\lambda)},
$$
which is the recursive relation \eqref{d-rec} with 
\begin{align*}
A_j^{(\mu,\beta)}  = \frac{(a_j^{(\mu,\beta)})^2\, b_{j+1}^{(\mu,\beta)} \,h_{j}^{(\mu+1,\beta)}}{b_{j}^{(\mu,\beta)}  h_{j-1}^{(\mu+1,\beta)}},\quad
 B_j^{(\mu,\beta)} + \lambda  C_j^{(\mu,\beta)}  =  
   \frac{(P_j^{(\mu, \beta)}, P_j^{(\mu, \beta)})_{\mu+1,\beta}}{b_{j}^{(\mu,\beta)} h_{j-1}^{(\mu+1,\beta)}}.
\end{align*}
Now, using the formula for $h_j^{(\mu+1,\b)}$ in \eqref{normJ} and the expressions for $a_j^{(\mu,\beta)}$
and $b_j^{(\mu,\beta)}$ in \eqref{coef-a-b}, it is easy to verify that $A_j^{(\mu,\beta)}$ has the explicit
expression in \eqref{A_j}. In order to evaluate $B_j^{(\mu,\beta)}$ and $C_j^{(\mu,\beta)}$, we need to calculate 
$(P_j^{(\mu, \beta)}, P_j^{(\mu, \beta)})_{\mu+1,\beta}$. From \eqref{sob-1-v},  
\begin{align} \label{eq:pj-snorm}
\left(P_{j}^{(\mu, \beta)}, P_j^{(\mu,\beta)}\right)_{\mu+1,\beta}   
    = & \int_{-1}^1 \left(P_{j}^{(\mu, \beta)}(t)\right)^2 w_{\mu+1,\beta}(t)dt  \\
& + {2} \lambda \left[A(\beta,d) \int_{-1}^1 \left(P_{j}^{(\mu, \beta)}(t)\right)^2 w_{\mu+1,\beta-1}(t)dt\right. \notag \\
& + 2\,B(\beta,d) \int_{-1}^1 P_{j}^{(\mu, \beta)}(t) \frac{d}{d t}P_j^{(\mu,\beta)}(t)w_{\mu+1,\beta}(t)dt \notag \\
& + \left.4 \int_{-1}^1 \left(\frac{d}{d t}P_{j}^{(\mu, \beta)}(t)\right)^2 w_{\mu+1,\beta+1}(t)dt\right]. \notag
\end{align}
Since $\lambda$ is a parameter, the first term in the right hand side is the numerator of $B_j^{(\mu,\beta)}$, 
whereas the expression in the square bracket is the numerator of $C_j^{(\mu,\beta)}$. 

Using \eqref{RAF} and the orthogonality of the Jacobi polynomials, we see that 
$$
 \int_{-1}^1 (P_{j}^{(\mu, \beta)}(t))^2 w_{\mu+1,\beta}(t)dt = \left(a_j^{(\mu,\beta)}\right)^2 
   h_j^{(\mu+1,\beta)} + \left(b_j^{(\mu,\beta)}\right)^2 h_{j-1}^{(\mu+1,\beta)},
$$
so that, by \eqref{normJ}  and  \eqref{coef-a-b} again, 
\begin{align*}
B_j^{(\mu,\beta)} & = \frac{(a_j^{(\mu,\beta)})^2}{b_j^{(\mu,\beta)}}\frac{h_{j}^{(\mu+1,\beta)}}{h_{j-1}^{(\mu+1,\beta)}} + b_j^{(\mu,\beta)} \\
&= \frac{(j+\mu+\beta+1)(j+\mu+1)(2j+\mu+\beta)}{j(2j+\mu+\beta+1)(2j+\mu+\beta+2)}+\frac{j+\beta}{2j+\mu+\beta+1}\\
&= 1+\frac{(\mu+\beta)(j+\mu+1)}{j(2j+\mu+\beta+2)}.
\end{align*}
For $C_j^{(\mu,\beta)}$, we compute the three terms in the square bracket of \eqref{eq:pj-snorm}. First
we combine the first term and half of the second term, and use the relation \eqref{jac-**} to deduce
\begin{align*}
 C_1  := &\, A(\beta,d) \int_{-1}^1 \left(P_{j}^{(\mu, \beta)}(t)\right)^2w_{\mu+1,\beta-1}(t)dt\\
 & \qquad + B(\beta,d) \int_{-1}^1 P_{j}^{(\mu, \beta)}(t)\frac{d}{d t}P_{j}^{(\mu, \beta)}(t)w_{\mu+1,\beta}(t)dt\\
  = & \, B(\beta,d)(\beta+j) \int_{-1}^1 P_{j}^{(\mu+1, \beta-1)}(t)\,P_j^{(\mu,\beta)}(t) w_{\mu+1,\beta-1}(t)dt\\
  = & \, B(\beta,d)(\beta+j) \frac{k_{j}^{(\mu, \beta)}}{k_{j}^{(\mu+1, \beta-1)}} h_{j}^{(\mu+1, \beta-1)} 
  = [2\beta -(d-2)](\beta+j) h_{j}^{(\mu+1, \beta-1)}.
\end{align*}
Next, by \eqref{derJ} and the orthogonality of the Jacobi polynomials, half of the second term becomes 
\begin{align*}
C_2:= &\, B(\beta,d) \int_{-1}^1 P_{j}^{(\mu, \beta)}(t)\frac{d}{d t}P_{j}^{(\mu, \beta)}(t)w_{\mu+1,\beta}(t)dt\\
=&\, B(\beta,d)\frac{j+\mu+\beta+1}{2}\int_{-1}^1 P_{j}^{(\mu, \beta)}(t) (1-t) P_{j-1}^{(\mu+1, \beta+1)}(t)w_{\mu,\beta}(t)dt\\
=&-B(\beta,d)\frac{j+\mu+\beta+1}{2}\frac{k_{j-1}^{(\mu+1, \beta+1)}}{k_{j}^{(\mu, \beta)}} h_j^{(\mu,\beta)} = -[2\beta -(d-2)] j \,h_j^{(\mu,\beta)}.
\end{align*}
Finally, using \eqref{derJ}, the third terms becomes 
$$
C_3: = 4 \int_{-1}^1 \left(\frac{d}{d t}P_{j}^{(\mu, \beta)}(t)\right)^2 w_{\mu+1,\beta+1}(t)dt
  =  (j+\mu+\beta+1)^2 h_{j-1}^{(\mu+1,\beta+1)}.
$$
Combining these terms we obtain 
$$
    C_j^{(\mu,\beta)} = \frac{C_1+C_2+C_3}{b_{j}^{(\mu,\beta)} \, h_{j-1}^{(\mu+1,\beta)}}   
$$
which simplifies to the formula \eqref{C_j}. 
\end{proof}

The recurrence relation \eqref{d-rec} shows that $d_j^{(\mu,\beta)}(\lambda)$ can be expressed as a 
continuous fraction. Consequently, we can express $d_j^{(\mu,\beta)}(\lambda)$ in terms of a rational 
function of $\lambda$ whose numerator and denominator are, respectively, the $(j-1)$th and $j$th 
elements of a sequence of orthogonal polynomials.

\begin{defn} 
For $j \in \NN_0$, define the polynomials $r_j^{(\mu,\b)}$ by $r_0^{(\mu,\b)}(\l) =1$ and 
\begin{equation} \label{3term-rj}
 r_{j+1}^{(\mu,\b)}(\lambda) = (C_j^{(\mu,\beta)} \lambda + B_j^{(\mu,\beta)})r_j^{(\mu,\b)}(\lambda) 
     - A_{j-1}^{(\mu,\beta)}r_{j-1}^{(\mu,\b)}(\lambda), 
\end{equation}
where we assume $r_{-1}^{(\mu,\b)}(\lambda) = 0$ and, for $j \ge 1$, $A_j^{(\mu,\beta)}$, $B_j^{(\mu,\beta)}$ and 
$C_j^{(\mu,\beta)}$ are defined in \eqref{A_j}, \eqref{B_j}, \eqref{C_j}, respectively, whereas for $j =0$, 
\begin{align*}
& A_0^{(\mu,\beta)}:=\frac{(\mu+1)(\beta+1)(\mu+\beta)}{(\mu+\beta+2)(\mu+\beta+3)},\qquad
   B_0^{(\mu,\beta)}:=\frac{(\mu+1)(\mu+\beta)}{(\mu+\beta+2)},\\
& C_0^{(\mu,\beta)}:=\frac{(\mu+1)(\mu+\beta)(2\beta-(d-2))}{2}.
\end{align*}
\end{defn}

For $\mu > -1$ and $\beta {>} (d-2)/2$, the coefficients $A_j^{(\mu,\beta)}$, $B_j^{(\mu,\beta)}$ and 
$C_j^{(\mu,\b)}$ are all positive for $ j \ge 0$. Consequently, by the Favard's theorem, the polynomials
$r_n^{(\mu,\b)}$ are orthogonal with respect to a positive linear functional. Since these coefficients
are explicitly known, one naturally asks if it is possible to identify these orthogonal polynomials, say 
with some classical orthogonal polynomials. The {numerator} of the coefficient $C_j^{(\mu,\b)}$ in \eqref{C_j} 
contains a quadratic polynomial in $j$ that does not factor into product of linear factors for independent
parameters of $\beta, \mu$ and $d$. This shows that the orthogonal polynomials are not hypergeometric
type in general. In some special cases, such as $3 (\mu+1) = 4 (\b - (d-2))$ or $\beta = (d-2)/2$, the 
quadratic term does factor and the corresponding orthogonal polynomials could be of hypergeometric 
type. For our study of the Sobolev polynomials on the ball in the next section, $\beta = n - 2j + \f{d-2}{2}$ 
for integers $n$ and $j$ with $0 \le j \le n/2$, the factorization happens only when $n - 2j =0$. Since this
is a sidetrack from our main purpose, we shall not pursue this direction any {further}. 

For our study, it is sufficient to remark that the three-term relation in \eqref{3term-rj} offers an effective 
way to generate $r_{j}^{(\mu,\b)}$. 

\begin{prop}
For $j = 0,1,2,\ldots$, the coefficients $d_j^{(\mu,\beta)}(\lambda)$ in \eqref{def-d} satisfy 
\begin{equation}\label{dj-induct}
d_j^{(\mu,\beta)}(\lambda) = - A_j^{(\mu,\beta)} \frac{r_j^{(\mu,\beta)}(\lambda)}{r_{j+1}^{(\mu,\beta)}(\lambda)}.
\end{equation}
\end{prop}

\begin{proof}
The case of $j = 0$ follows directly from \eqref{d-cero}, which allows us to deduce, from 
$r_0^{(\mu,\beta)}(\lambda) =1$ and 
$
   r_1^{(\mu,\beta)}(\lambda) = C_0^{(\mu,\beta)}\lambda + B_0^{(\mu,\beta)},
$
the formulas for $ A_0^{(\mu,\beta)}$, $B_0^{(\mu,\beta)}$ and $C_0^{(\mu,\beta)}$. Assume that 
\eqref{dj-induct} has been established for integers up to $j-1$. Then, by \eqref{d-cero}, 
\begin{align*}
d_{j}^{(\mu,\beta)}(\lambda) =& -\frac{A_j^{(\mu,\beta)}}{(B_j^{(\mu,\beta)} + \lambda C_j^{(\mu,\beta)}) + 
    d_{j-1}^{(\mu,\beta)}} \\
  = & -\frac{A_j^{(\mu,\beta)}}{(B_j^{(\mu,\beta)} + \lambda C_j^{(\mu,\beta)}) - A_{j-1}^{(\mu,\beta)}\frac{\displaystyle{r_{j-1}^{(\mu,\beta)}(\lambda)}}{\displaystyle{r_{j}^{(\mu,\beta)}(\lambda)}}}\\
 =& -\frac{A_j^{(\mu,\beta)}r_j^{(\mu,\beta)}(\lambda)}{(B_j^{(\mu,\beta)} + \lambda C_j^{(\mu,\beta)})r_j^{(\mu,\beta)}(\lambda) - A_{j-1}^{(\mu,\beta)}r_{j-1}^{(\mu,\beta)}(\lambda)},
\end{align*}
which is \eqref{dj-induct}. By induction, this completes the proof. 
\end{proof}

\begin{cor}
Let $\mu +1> -1$ and $\beta {>} \max\{0,(d-2)/2\}$. Then, for $j\ge 0$, the Sobolev orthogonal polynomials 
$q_j^{(\mu,\b)}$ in \eqref{qj-defn} satisfy 
\begin{equation}\label{qj-Jacobi}
  q_j^{(\mu+1,\beta)}(t) = \frac{1}{a_j^{(\mu,\beta)}r_j^{(\mu,\beta)}(\l)} 
     \sum_{i =0}^j \frac{D_i^{(\mu, \beta)}}{a_i^{(\mu, \beta)}} r_i^{(\mu,\beta)}(\l)P_i^{(\mu, \beta)}(t),
\end{equation}
where $a_j^{(\mu, \beta)}$ are given in \eqref{coef-a-b}, $r_j^{(\mu,\beta)}(\l)$ are given in \eqref{3term-rj}, 
and $D_j^{(\mu, \beta)}$ are defined by $D_0^{(\mu, \beta)} = (\mu + \b)(\mu+\b+1)$ and, for $j \ge 1$, 
\begin{equation}\label{Dj-defn}
D_j^{(\mu, \beta)} : = \frac{ 2^j (j+\mu+\b +1)(2j+\mu+\b) (\frac{\mu+\b+1}{2})_j (j-1)!}{(\mu+1)_j (\b+1)_j}.
\end{equation}
Furthermore, the {Sobolev norm} of $q_j^{(\mu,\b)}$ is given by 
\begin{equation}\label{norm-qj}
  \left(q_j^{(\mu+1,\b)},  q_j^{(\mu+1,\b)} \right)_{\mu+1,\b}  = 
       \frac{b_{j+1}^{(\mu,\b)}  h_j^{(\mu+1,\b)} } {  A_j^{(\mu,\b)}} \frac{ r_{j+1}^{(\mu,\b)}(\l)}{r_{j}^{(\mu,\b)}(\l)},
\end{equation}
where $b_j^{(\mu, \beta)}$, $h_j^{(\mu+1,\b)}$ and $A_j^{(\mu,\b)}$ are given in \eqref{coef-a-b}, \eqref{normJ} 
and \eqref{A_j}, respectively. 
\end{cor}

\begin{proof}
Using \eqref{dj-induct}, the relation \eqref{1-2-one-v} can be rewritten as 
$$
  q_j^{(\mu+1,\beta)}(t) = \frac{1}{a_j^{(\mu,\beta)}} P_j^{(\mu, \beta)}(t) + 
    \frac{A_{j-1}^{(\mu,\beta)}}{a_j^{(\mu,\beta)}}  \frac{r_{j-1}^{(\mu,\beta)}(\l)} {r_{j}^{(\mu,\beta)}(\l)} 
      q_{j-1}^{(\mu+1,\beta)}(t).
$$
For $j \ge 2$, it is easy to check that ${A_{j-1}^{(\mu,\beta)}}/ {a_j^{(\mu,\beta)}} = 
{D_{j-1}^{(\mu,\beta)}}/{D_j^{(\mu,\beta)}}$, so that the above identity can be written as
$$
 \wt q_j^{(\mu+1,\beta)}(t) = \frac{D_j^{(\mu, \beta)}}{a_j^{(\mu,\beta)}}r_j^{(\mu,\beta)}(\l) P_j^{(\mu, \beta)}(t) + 
         \wt q_{j-1}^{(\mu+1,\beta)}(t),
$$
where $\wt q_j^{(\mu,\beta)}(t) = D_j^{(\mu,\beta)}r_j^{(\mu,\beta)}(\lambda) q_j^{(\mu+1,\beta)}(t)$. 
Furthermore, setting $j  =1$ in \eqref{1-2-one-v} shows that the above relation also holds for $j =1$
(which is how the value of $D_0^{(\mu,\b)}$ is determined). In particular, it follows that 
$\wt q_0^{(\mu,\beta)}(t) = D_0^{(\mu,\beta)}$. Summing up the telescoping sequence shows that 
 $\wt q_j^{(\mu+1,\beta)}(t)$ can be written as a sum in terms of the Jacobi polynomials, which 
proves \eqref{qj-Jacobi}. 

Finally, \eqref{norm-qj} follows directly from \eqref{def-d} and \eqref{dj-induct}. 
\end{proof}

\section{Sobolev inner product on the unit ball}
\setcounter{equation}{0}

Recall that our main task is to consider orthogonal polynomials with respect to the inner 
product \eqref{eq:main-ip}, which we restate below. 

\begin{defn} 
Let $\l > 0$ and $\mu > -1$. For $f, g \in \Pi^d$, we define 
\begin{equation*}
\la f,g\ra _{\nabla, W_\mu, \BB^d} := b_\mu \left[\int_{\BB^d} f(x) g(x)  W_\mu(x)  dx +
    \lambda  \int_{\BB^d} \nabla f(x)\,\cdot \nabla g(x) W_\mu(x) dx\right],
\end{equation*}
\end{defn}

It is easy to see that $\la \cdot,\cdot \ra_{\nabla, W_\mu, \BB^d}$ is indeed an inner product. Let 
$\mathcal{V}_n^d(\nabla, W_\mu, \BB^d)$ denote the linear space of polynomials of total degree 
$n$ that are orthogonal to all polynomial of lower {degree} with respect to this inner product. It 
follows then that $\dim \, \mathcal{V}_n^d(\nabla, W_\mu, \BB^d) = r_n^d.$

An orthogonal basis for $\CV_n^d(\nabla, W_\mu, \BB^d)$ is given in the following theorem. 

\begin{thm}
Let $\lambda >0$. For $0 \le j \le n/2$, let $\b_j: = n-2j+ \f{d-2}2$ and let $q_k^{(\mu, \b_j)}(t)$ be the $k$--th
Sobolev orthogonal polynomial associated with the inner product $(\cdot, \cdot)_{\mu, \b_j}$. 
Let $\{Y_\nu^{n-2j}: 1 \le \nu \le a_{n-2j}^d\}$ be an orthonormal basis of $\CH_{n-2j}^d$. Define 
\begin{equation}\label{baseR}
   R_{j,\nu}^{n}(x) := q_{j}^{(\mu, \b_j)}(2 \|x\|^2 -1) Y_\nu^{n-2j}(x).
\end{equation}
Then the set $\{R^n_{j,\nu}(x):  0\le j \le n/2, \quad 0\le \nu \le a_{n-2j}^d\}$ is a mutually orthogonal basis of
$\mathcal{V}_n^d(\nabla, W_\mu, \BB^d)$. Moreover, 
$$
  \la R_{j,\nu}^n, R_{j,\nu}^n \ra_{\nabla, W_\mu, \BB^d} =  \frac{\Gamma(\mu+1+ \frac d 2)}{
     \Gamma(\mu+1)\Gamma(\frac d 2)2^{\b_j +\mu} } 
      \left (q_j^{(\mu,\b_j)}, q_j^{(\mu,\b_j)}\right)_{\mu, \b_j}.
$$
\end{thm}

\begin{proof}
We need to {calculate}, 
\begin{align*}
  \la R_{j,\nu}^{n}(x), R_{k,\eta}^{m}(x)\ra_{\nabla, W_\mu, \BB^d} 
      = & b_\mu \left[\int_{\BB^d} R_{j,\nu}^{n}(x) R_{k,\eta}^{m}(x)W_\mu(x) dx \right. \\
    & \quad + \lambda \left.\int_{\BB^d} \nabla R_{j,\nu}^{n}(x)\cdot \nabla R_{k,\eta}^{m}(x) W_\mu(x) \, dx\right].
\end{align*}
We consider the two integrals in the right hand side separately. For the first integral, using \eqref{changevar} 
and the fact that $Y_\nu^{n-2j}$ is homogenous and orthonormal {with} respect to $\la \cdot,\cdot \ra_{\sph}$,
we obtain, 
\begin{align*}
& b_\mu \int_{\BB^d} R_{j,\nu}^{n}(x) R_{k,\eta}^{m}(x)W_\mu(x) dx \\
  & \quad = b_\mu \o_d \int_0^1 q_j (2r^2 -1) q_k (2r^2 -1) r^{2 \b_j +1}  
       (1-r^2)^{\mu}\delta_{n-2j,m-2k}\delta_{\nu,\eta} \\
   & \quad =  \frac{b_\mu \o_d}{2^{n-2j+\mu+d/2+1}}
      \left[\int_{-1}^1 q_{j}  (t)\,q_{k} (t)(1-t)^{\mu},(1+t)^{n-2j+\frac{d-2}{2}}dt \right]\delta_{n-2j,m-2k}\delta_{\nu,\eta},
\end{align*}
where we have omitted the superscript of $q_j^{(\mu,n-2j+\f{d-2}{2})}$ for simplicity, which we shall adopt
in the rest of this proof. 

For the second integral, we observe that, by the product rule of differentiation, 
$$
  \partial_i Q_{j,\nu}^{n}(x; W_\mu) = q'_{j}(2\,\|x\|^2 -1)\, 4\, x_i\, Y_\nu^{n-2j}(x) 
      + q_{j}(2\,\|x\|^2 -1)\, \partial_i Y_\nu^{n-2j}(x),
$$
which implies that 
\begin{align*}
  \nabla R_{j,\nu}^{n}(x)\cdot\nabla R_{k,\eta}^{m}(x)
= &   q_{j}(2\,\|x\|^2 -1)\,q_{k}(2\,\|x\|^2 -1)\, \nabla Y_\nu^{n-2j}(x)\cdot \nabla Y_\eta^{m-2k}(x)\\ 
& + 4 (m-2k) q'_{j}(2\|x\|^2 -1) q_{k}(2 \|x\|^2 -1) Y_\nu^{n-2j}(x) Y_\eta^{m-2k}(x) \\
& + 4 (n-2j) q_{j}(2\|x\|^2 -1) q'_{k}(2\|x\|^2 -1) Y_\nu^{n-2j}(x)  Y_\eta^{m-2k}(x) \\
& + 16  \|x\|^2 q'_{j}(2\,\|x\|^2 -1)q'_{k}(2\|x\|^2 -1) Y_\nu^{n-2j}(x) Y_\eta^{m-2k}(x).
\end{align*}
We now integrate above expression term by term and apply \eqref{changevar}. For the first term,
we obtain, using \eqref{nabla-Y-int}, that 
\begin{align*}
& \int_{\BB^d} q_{j}(2\|x\|^2 -1)q_{k}(2\|x\|^2 -1) \nabla Y_\nu^{n-2j}(x)\cdot \nabla Y_\eta^{m-2k}(x)W_\mu(x)dx\\
& \qquad = \int_{0}^1 q_{j}(2r^2 -1)q_{k}(2r^2 -1) (1-r^2)^{\mu}r^{n-2j+m-2k+d-1}\\
& \qquad \qquad \times \frac{(n-2j)(2(n-2j)+d-2)}{r^2} dr \omega_{d}\delta_{n-2j,m-2k}\delta_{\nu,\eta} \\
& \qquad = \left[\int_{-1}^1 q_{j}(t)\,q_{k}(t)(1-t)^{\mu} (1+t)^{n-2j+d/2-2}\,dt \right]\\
& \qquad \qquad \times \frac{(n-2j)(2(n-2j)+d-2)\,\omega_{d}} {2^{n-2j+\mu+d/2}}\delta_{n-2j,m-2k}\delta_{\nu,\eta}.
\end{align*}
The second and the third terms are similar. For the second term,  we obtain
\begin{align*}
& 4(m-2k)  \int_{\BB^d} q'_{j}(2\|x\|^2 -1)q_{k}(2\|x\|^2 -1) Y_\nu^{n-2j}(x) Y_\eta^{m-2k}(x)W_\mu(x)dx\\
 =\ & 4(m-2k)\int_{0}^1 q'_{j}(2r^2 -1)\,q_{k}(2r^2 -1) (1-r^2)^{\mu}r^{2n-4j+d-1}dr  \omega_{d}\delta_{n-2j,m-2k}\delta_{\nu,\eta}\\
 =\ & \int_{-1}^1 q'_{j}(t) q_{k}(t)(1-t)^{\mu}(1+t)^{n-2j+d/2-1}dt  \frac{2(m-2k)\,\omega_{d}}{2^{n-2j+\mu+d/2}}\delta_{n-2j,m-2k}\delta_{\nu,\eta}.
\end{align*}
Finally, for the fourth term, we obtain
\begin{align*}
&16   \int_{\BB^d} \|x\|^2 q'_{j}(2\|x\|^2 -1)q'_{k}(2\|x\|^2 -1)Y_\nu^{n-2j}(x) Y_\eta^{m-2k}(x)W_\mu(x)dx \\
=& 16 \int_{0}^1 q'_{j}(2r^2 -1)\,q'_{k}(2r^2 -1) (1-r^2)^{\mu}r^{2n-4j+d+1}dr 
    \omega_{d}\delta_{n-2j,m-2k}\delta_{\nu,\eta}\\
= & \int_{-1}^1 q'_{j}(t)q'_{k}(t)(1-t)^{\mu} (1+t)^{n-2j+d/2} dt
   \frac{\omega_{d}}{2^{n-2j+\mu+d/2-2}}\delta_{n-2j,m-2k}\delta_{\nu,\eta}.
\end{align*}

Putting all four terms together, we conclude that the second integral satisfies 
\begin{align*}
& \int_{\BB^d} \nabla R_{j,\nu}^{n}(x)\cdot \nabla R_{k,\eta}^{m}(x) W_\mu(x) dx =
      \frac{\omega_{d-1}}{2^{n-2j+\mu+d/2}}\,\delta_{n-2j,m-2k}\,\delta_{\nu,\eta}\\
& \quad \times \left[(n-2j)(2(n-2j)+d-2)\int_{-1}^1 q_{j}(t)q_{k}(t)(1-t)^{\mu}(1+t)^{n-2j+d/2-2}dt \right.\\
& \qquad + 2 (n-2j) \int_{-1}^1 (q_{j}(t)q_{k}(t))'(1-t)^{\mu} \,(1+t)^{n-2j+d/2-1}dt\\
&\qquad + \left. 4\int_{-1}^1 q'_{j}(t)q'_{k}(t)(1-t)^{\mu} (1+t)^{n-2j+d/2}dt\right].
\end{align*}
Together, the first and the two integrals lead to 
\begin{align*}
  \la R_{j,\nu}^{n}, R_{k,\eta}^{m} \ra_{\nabla, W_\mu, \BB^d} = \frac{b_\mu \omega_d}{2^{n-2j+\mu+d/2}}            
     \delta_{n-2j,m-2k}\delta_{\nu,\eta} (q_j, q_k)_{\mu, n-2j+\f{d-2}{2}}. 
\end{align*}
Using \eqref{b-mu-omega}, this completes the proof. 
\end{proof}
 
By \eqref{norm-qj}, the norm of $R_{j,\nu}^n$ can be {expressed} in terms of $r_i^{\mu, \b_j}$, which can
be computed recursively. Moreover, as a corollary of \eqref{1-2-one-v}, we also have the following relation 
between classical orthogonal polynomials $P_{j,\nu}^n$ on the ball and our Sobolev orthogonal polynomials. 
Let us denote $P_{j,\nu}^n$ in $\CV_n^d(W_\mu)$ by $P_{j,\nu}^n(\cdot;W_\mu)$ and, similarly, denote 
$R_{j,\nu}^n(x)$ in $\CV_n^d(\nabla, W_\mu, \BB^d)$ by $R_{j,\nu}^n(x;W_\mu)$ to emphasis the 
dependence on $\mu$. 

\begin{cor}
Let $\mu > -1$. For $0 \le j \le n/2$, let $\b_j = n-2j + \f{d-2}{2}$. Then 
\begin{equation*}
P_{j,\nu}^{n}(x; W_\mu) = a_j^{(\mu,\b_j)} R_{j,\nu}^{n}(x; W_{\mu+1}) + d_{j-1}^{(\mu,\b_j)}(\l)
     R_{j-1,\nu}^{n-2}(x; W_{\mu+1}).
\end{equation*}
Furthermore, 
\begin{equation*} 
 R_{j,\nu}^{n}(x; W_{\mu+1}) = \frac{1}{a_j^{(\mu,\beta)} r_j^{(\mu,\beta)}(\l)} 
     \sum_{i =0}^j \frac{D_i^{(\mu, \beta)}}{a_i^{(\mu, \beta)}} r_i^{(\mu,\beta)}(\l)P_{i,\nu}^{n}(x; W_\mu),
\end{equation*}
where the notations are same as those in \eqref{qj-Jacobi}.
\end{cor}

\end{document}